\documentclass[11pt]{article}
\usepackage{amsmath}
\usepackage{amssymb}
\usepackage{amsthm}
\usepackage{amscd}
\usepackage{amsfonts,amsthm}
\usepackage{fontenc}
\usepackage{graphicx,psfrag}
\usepackage{subcaption}
\usepackage[T1]{fontenc}
\usepackage[utf8]{inputenc}
\usepackage{psfrag}
\usepackage{url}
\usepackage[all]{xy}
\usepackage{indentfirst}
\usepackage{epstopdf}
\usepackage[
pdfauthor={CS},
pdftitle={Halin Graphs},
pdfstartview=XYZ,
bookmarks=true,
colorlinks=true,
linkcolor=blue,
urlcolor=blue,
citecolor=blue,
bookmarks=true,
linktocpage=true,
hyperindex=true
]{hyperref}

\usepackage[margin=1.1in]{geometry}

\numberwithin{equation}{section}

\newcounter{AbcT}

\newcommand{\nc}{\newcommand}
\nc{\meet}{\wedge}
\nc{\op}{\operatorname}\nc{\FP}{\op{FP}}\nc{\FS}{\op{FS}}\nc{\FPhat}{\widehat{\op{FP}}}

\newtheorem {Theorem}    {Theorem}[section]

\newtheorem {Lemma}      [Theorem]    {Lemma}
\newtheorem {Corollary}   [Theorem] {Corollary}

\theoremstyle{remark}

\nc{\N}{\mathbb{N}}

\newcounter{DM@bibnum}

\nc{\set}[2]{\{#1 \,:\, #2\}}

\def\NSL_2{{\mathcal N SL_2}}

\def\phi{\varphi}

\def\hbar{\bar h}

\begin{document}
%\includepdf[pages={1-15}]{<pdf-file>}

\title{Nonempty intersection of longest paths in $2K_2$-free graphs}
\author{Gili Golan and Songling Shan\\
{ \small \smallskip Vanderbilt University, Nashville, TN\,37240}}
\date{}

\maketitle
\abstract{In 1966, Gallai asked {\it whether all longest paths in a connected graph share a common vertex}.
Counterexamples indicate that this is not true in general. However, Gallai's question is positive for certain well-known classes of
connected graphs, such as split graphs, interval graphs, circular arc graphs, outerplanar graphs,
and series-parallel graphs. A graph is $2K_2$-free if it does not contain two independent edges as an induced subgraph. In this paper, we show that in nonempty $2K_2$-free graphs, every vertex of maximum degree is common to all longest paths.  Our result implies that all longest paths in a nonempty $2K_2$-free graph have a nonempty intersection. In particular, it gives a new proof for the result on split graphs, as split graphs are $2K_2$-free.
\vspace{2mm}

% \%emph\vspace{2mm}
\vspace{2mm}

\section{Introduction}

All graphs considered in this paper are finite and simple. A path $P$ in a graph $G$ is a \emph{longest path} in $G$ if there is no path in $G$ strictly longer than $P$. The study of intersections of longest paths in  graphs has a long history. It is well-known that in a connected
graph, any two longest paths have a vertex in common~\cite{ore}. It
is also well-known that a family of pairwise intersecting subtrees of a tree
has nonempty intersection.  This motivated Gallai in 1966 to ask whether all longest paths in a connected graph have a vertex in common. In 1974, Walther~\cite{Walther69} gave a counterexample to Gallai's problem. Walther's example was a graph with 25 vertices. A counterexample with 12 vertices was constructed by Walther and Voss~\cite{Walther74} and independently by Zamfirescu~\cite{Zamfirescu76} (see Figure~\ref{fig:wzamf}).
Brinkmann and van Cleemput \cite{brink} showed that for graphs with less than 12 vertices, Gallai's problem has a positive solution.

\begin{figure}[!htb]
	\centering
	\includegraphics[scale=0.1]{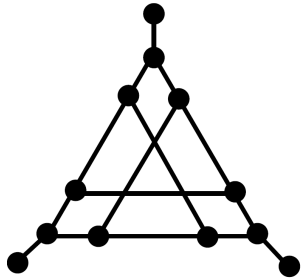}
	\caption{The smallest counterexample to Gallai's question.}		
	\label{fig:wzamf}
\end{figure}

Counterexamples to the problem were found for 2- and 3-connected planar graphs \cite{L-P-2-3-connected-planar-counterexample} as well as for non-planar graphs \cite{L-P-connected-planar-counterexample}.
Every hypo-traceable graph (i.e., a non-Hamiltonian graph where all vertex-deleted subgraphs are Hamiltonian) is clearly a counterexample. As there are infinitely many hypo-traceable graphs\,(see Thomassen~\cite{Thomassen76}),
 there are infinitely many counterexamples to the problem. 

Further research on intersection of longest paths in graphs is carried out in two main directions. The first direction is concerned with the intersection of a fixed number of longest paths. The most interesting open problem in this area is whether any three longest paths in a connected graph share a common vertex. The question was first asked by Zamfirescu~\cite{Voss91}. It is also mentioned in \cite{BCC95,HarrisHM2008,West, Zamfirescu01}. Progress in this direction was made by de~Rezende et al.~\cite{deRezendeFMW13}, who proved that if all nontrivial blocks of a connected graph are Hamiltonian, then any three longest paths in the graph share a vertex. Skupie\'n~\cite{Skupien96} showed that for every \(p \geq 7\), there exists a connected graph such that~\(p\) longest paths have no common vertex and every~\(p-1\) longest paths have a common vertex.

The second direction considers Gallai's problem for particular classes of graphs. Klav\v{z}ar and Petkov\v{s}ek~\cite{KlavzarP90} gave an affirmative answer to Gallai's question for split graphs and for cacti (a graph is a \emph{cactus} if and only if each of its blocks is a simple cycle or a vertex or a single edge).
An affirmative answer was given for the class of circular-arc graphs by Balister et al.~\cite{BalisterGLS04} (see also Joos~\cite{Joos15} for a correction of a gap in the proof from \cite{BalisterGLS04}). A positive answer for connected outerplanar graphs and 2-trees was given by de Rezende et al.~\cite{Eurocomb11}. Recently, the second author with Chen et al.~\cite{Chen2017287} extended this result, giving a positive solution to Gallai's problem for the class of connected series-parallel graphs.

In this paper, we investigate the intersection of all longest paths in connected $2K_2$-free graphs.
A graph is {\it $2K_2$-free} if it contains no two independent edges as an induced subgraph. The class of $2K_2$-free graphs is well studied, for instance, see~\cite{
2k2-tough, CHUNG1990129, MR845138, MR2279069,2k21, 1412.0514,MR1172684}.
It is a superclass of the class of {\it split graphs\/},
where vertices can be partitioned into a clique and an independent set.
One can also easily check that every {\it cochordal\/} graph (i.e., a graph that is the complement of a
chordal graph) is $2K_2$-free and so the class of $2K_2$-free graphs is
at least as rich as the class of chordal graphs. We give a positive solution to Gallai's problem for the class of $2K_2$-free graphs. In fact, we prove the following stronger result. Recall that a graph is {\it nonempty} if it contains at least one edge. 

\begin{Theorem}\label{thm1}
In a nonempty $2K_2$-free graph, every vertex of maximum degree is common to all longest paths.
\end{Theorem}

As a $2K_2$-free graph contains at most one 
component which contains an edge, we replaced the ``connectivity'' assumption 
by the ``non-emptiness'' assumption. 
 In particular, we get an alternative proof for Klav\v{z}ar and Petkov\v{s}ek's result for split graphs ~\cite{KlavzarP90}.

\begin{Corollary}\label{thm2}
If $G$ is a nonempty  split graph or cochordal graph,
then every vertex of maximum degree is common to all longest paths.
\end{Corollary}

\vskip .2cm

For a graph $G$ we will denote by $V(G)$ and $E(G)$ the vertex set and edge
set of $G$, respectively. If $uv\in E(G)$, we write $u\sim v$ to denote the adjacency of $u$ and $v$.
For two disjoint subsets $S, T\subseteq V(G)$, we denote by $E_G(S,T)$ the set of edges of $G$ with one end in $S$ and the other in $T$. If $u\in V(G)$ we denote by $N_G(u)$ the set of neighbors of $u$ in $G$. If $G$ is clear from the context, we omit the subscript $G$ and write $E(S,T)$ and $N(u)$.

\section{Proof of Theorem~\ref{thm1}}

In this section we prove Theorem~\ref{thm1}.
We will need the following three lemmas. A path $P$
in a graph $G$ is {\it dominating} if $G-V(P)$
is edgeless.

\begin{Lemma}\label{dom}
Let $G$ be a  $2K_2$-free graph. Then every longest path in $G$ is dominating.
\end{Lemma}

\begin{proof}
Let $P$ be a longest path in $G$. Assume by contradiction that $P$ is not dominating. Then there exists an edge $uv\in E(G)$ such that $u,v\notin V(P)$. Let $v_0v_1$ be the first edge of the path $P$. Since $G$ is $2K_2$-free, there must be an edge $e'$ in $G$ which connects the edge $uv$ to the edge $v_0v_1$. Without loss of generality, we can assume that $e'$ connects $v$ to either $v_0$ or $v_1$. If $e'=vv_0$ then $uvv_0v_1\cdots v_{\ell}$ is a path in $G$ longer than $P$. If $e'=vv_1$ then $uvv_1\cdots v_{\ell}$ is a path in $G$ longer than $P$.
\end{proof}

\begin{Lemma}\label{4}
Let $G$ be a  $2K_2$-free graph. Let $P=v_0v_1\cdots v_{\ell}$ be a longest path in $G$ and let $x$ be a vertex of $G$ which does not belong to $P$. Then the following assertions hold.
\begin{enumerate}
\item[(1)] The vertex $x$ is not adjacent to the endpoints $v_0$ and $v_{\ell}$ of $P$.
\item[(2)] The vertex $x$ does not have two neighbors which are consecutive vertices $v_i,v_{i+1}$ on $P$.
\item[(3)] If $v_a$ is a neighbor of $x$ then $v_0$ is not adjacent to $v_{a+1}$.
\item[(4)] If $v_a$ and $v_b$ are distinct neighbors of $x$ then $v_{a+1}$ is not adjacent to $v_{b+1}$.
\end{enumerate}
\end{Lemma}

\begin{proof}
If (1) or (2) does not hold then one can modify the path $P$ to get a longer path in $G$.
To prove assertion (3), assume by contradiction that $v_0\sim v_{a+1}$. Then
$$Q=xv_av_{a-1}\cdots v_1v_0v_{a+1}v_{a+2}\cdots v_{\ell}$$ is a path in $G$
 which contains all the vertices of $P$ and the vertex $x$, in contradiction to $P$ being a longest path in $G$.

For part (4), assume that $v_{a+1}\sim v_{b+1}$ and that $a<b$. Then $$Q=v_0v_1\cdots v_axv_bv_{b-1}\cdots v_{a+1}v_{b+1}\cdots v_{\ell}$$ is a path in $G$ longer than $P$.
\end{proof}

The following lemma was first proved in~\cite[Theorem 1]{CHUNG1990129}. To make the paper self-contained, we recall the proof.

\begin{Lemma}\label{bi}
Let $G$ be a $2K_2$-free graph and let $S\subseteq V(G)$ be an independent set. Let $T\subseteq V(G)-S$. Then there exists $y\in T$ such that $N(y)$ meets all edges in $E(S,T)$.
\end{Lemma}

\begin{proof}
Let $G'$ be the bipartite subgraph of $G$ with partite sets $S$ and $T$ and edge set $E_G(S,T)$.
Let $y\in T$ be a vertex of maximum degree in $G'$. Let $S'=N_G(y)\cap S$ and $T'=N_G(y)\cap T$.
We claim that $N_G(y)$ meets all edges in $E_G(S,T)$. Otherwise, let $e=uv\in E_G(S,T)$
be such that $u\in S$, $v\in T$ and $u,v\notin N_G(y)$. To get a contradiction we show that $d_{G'}(v)>d_{G'}(y)$.
Indeed, let $s\in S'$. Then $sy\in E_G(S,T)$. Consider the edges $uv$ and $sy$ in $G$. Since $S$ is independent, $u\not\sim s$. Since $u,v\notin N_G(y)$ and $G$ is $2K_2$-free, $v\sim s$ in $G$. It follows that $v$ is adjacent to every vertex in $S'$ and to the vertex $u\in S-S'$. Hence $d_{G'}(v)\ge|S'|+1>|S'|=d_{G'}(y)$. %, in contradiction to the choice of $t$.
\end{proof}

We are now ready to prove Theorem~\ref{thm1}.

\begin{proof}[Proof of Theorem \ref{thm1}]
Let $G$ be a nonempty  $2K_2$-free graph and let $P=v_0v_1\cdots v_{\ell}$ be a longest path in $G$. Assume that $x\in V(G)$ is a vertex of maximum degree in $G$ which does not belong to $P$. Let $k=d(x)=\Delta(G)$. By Lemma~\ref{dom}, $N(x)\subseteq V(P)$. Let
$$S=\{v_0, v_{a+1}\mid v_a\in N(x)\}\subseteq V(P).$$
 By lemma~\ref{4}(3,4), $S$ is an independent set. Let $T=V(P)-S$. By Lemma~\ref{4}(1,2), $V(P)$ contains at least $2k+1$ vertices.
%If $|V(p)|=2k+1$, then
The set $\{v_av_{a+1}\mid v_a\in N(x)\}$ is a set of $k$ independent edges in $E(S,T)$ (i.e., $k$ edges which pairwise do not share an endpoint).

We claim that if $|V(P)|\ge 2k+2$ then there are $k+1$ independent edges in $E(S,T)$. Indeed, the $k$ neighbors of $x$ separate $P$ into $k+1$ non-trivial subpaths (see Lemma \ref{4}(1)). By the pigeonhole principle one of these subpaths contains at least two vertices in $V(P)-N(x)$. If $v_0$ is an endpoint of this subpath, then $v_0,v_1\notin N(x)$ and
$$\{v_0v_1,v_av_{a+1}\mid v_a\in N(x)\}\subseteq E(S,T)$$
is an independent subset of $k+1$ edges.
If $v_{\ell}$ is an endpoint of this subpath then
$$\{v_0v_1,v_{a+1}v_{a+2}\mid v_{a}\in N(x)\}\subseteq E(S,T)$$
is an independent subset of size $k+1$.
Thus, we can assume that the endpoints of this subpath are $v_\alpha,v_\beta\in N(x)$ for some
$\alpha<\beta$ in $\{1,\dots,\ell-1\}$. Then
$$\{v_0v_1\}\cup\{v_{a+1}v_{a+2}\mid a\le\alpha,\ v_a\in N(x) \}\cup\{ v_bv_{b+1}\mid b\ge\beta,\ v_b\in N(x)\}\subseteq E(S,T)$$
is a set of $k+1$ independent edges.

Now, by Lemma~\ref{bi}, there is a vertex $y\in T$ such that $N(y)$ meets all edges in $E(S,T)$. If $|V(P)|\ge 2k+2$, then $y$ has at least $k+1$ neighbors in $V(P)=S\cup T$ since $E(S,T)$ contains an independent set of $k+1$ edges. Then $d(y)\ge k+1>k=\Delta(G)$, a contradiction. If $|V(P)|=2k+1$ then $T=N(x)$. Indeed, the disjoint union $S\cup N(x)\subseteq V(P)$ and $|S|=k+1$, $|N(x)|=k$. Hence $N(x)=V(P)-S=T$. In particular, in that case, $y\in N(x)$. Since $N(y)$ meets all edges in $E(S,T)$ and $E(S,T)$ contains an independent set of $k$ edges, $y$ has at least $k$ neighbors in $V(P)$. Since $x$ is also a neighbor of $y$ we have $d(y)\ge k+1>k=\Delta(G)$, a contradiction.
\end{proof}

\textbf{{\noindent \large Acknowledgements}}

The authors would like to  thank  Guantao Chen  and Akira Saito
for their helpful discussions.

\bibliographystyle{amsplain} 	
\bibliography{longest}	 		 	 	 %Name der .bib Datei mit der bibliographischen DB
\clearpage

%\begin{thebibliography}{AAA}
%
%\bibitem{CGTT} F. R. K Chung, A. Gy\'arf\'as, Z. Tuza and W. T. Trotter,
%\it The maximum number of edges in $2K_2$-free graphs of bounded degree,
%\rm Discrete Mathematics 81 (1990) 129-135.
%
%
%\end{thebibliography}

\end{document}